\documentclass[12pt]{amsart}
\usepackage{geometry}
\usepackage{graphicx}
\usepackage{amssymb}
\usepackage{epstopdf}
\usepackage{amsmath}
\usepackage{color}
\usepackage{hyperref}
\usepackage{pdfsync}
 \usepackage{tikz}
 \usepackage{tikz-cd}

\usepackage[all]{xy}
\xyoption{matrix}
\xyoption{arrow}

\usetikzlibrary{arrows,decorations.pathmorphing,backgrounds,positioning,fit,petri}

\tikzset{help lines/.style={step=#1cm,very thin, color=gray},
help lines/.default=.5} 
\tikzset{thick grid/.style={step=#1cm,thick, color=gray},
thick grid/.default=1} 

\newtheorem{theorem}{Theorem}[section]

\newtheorem{corollary}[theorem]{Corollary}
\newtheorem{proposition}[theorem]{Proposition}

\theoremstyle{definition}
\newtheorem{definition}[theorem]{Definition}
\newtheorem{example}[theorem]{Example}

\theoremstyle{remark}
\newtheorem{remark}[theorem]{Remark}
\numberwithin{equation}{section}

\DeclareMathOperator{\Ker}{Ker}
\DeclareMathOperator{\Ima}{Im}
\DeclareMathOperator{\Hom}{Hom}
\DeclareMathOperator{\id}{id}

\DeclareMathOperator{\dist}{dist}

\DeclareMathOperator{\rank}{rank}

\newcommand{\commentout}[1]{}

\newcommand{\cA}{\ensuremath{{\mathcal{A}}}}

\newcommand{\cM}{\ensuremath{{\mathcal{M}}}}

\newcommand\del{\partial}

\begin{document}
\title{Scaled Homology and Topological Entropy}
\author{Bingzhe Hou}
\address{Bingzhe Hou, School of Mathematics, Jilin university, Changchun, P.R.China, 130012}
\email{houbz@jlu.edu.cn}
\author{Kiyoshi Igusa}
\address{Kiyoshi Igusa, Department of Mathematics, Brandeis university, Waltham, MA, 02453}
\email{igusa@brandeis.edu}
\author{Zihao Liu}
\address{Zihao Liu, Department of Mathematics, Brandeis university, Waltham, MA, 02453}
\email{zihaoliu@brandeis.edu}

\begin{abstract}
In this paper, we build up a scaled homology theory, $lc$-homology, for metric spaces such that every metric space can be visually regarded as ``locally contractible" with this newly-built homology. We check that $lc$-homology satisfies all Eilenberg-Steenrod axioms except the exactness axiom whereas its corresponding $lc$-cohomology satisfies exactness axiom for cohomology. This homology can relax the smooth manifold restrictions on the compact metric space such that the entropy conjecture will hold for the first $lc$-homology group.
\end{abstract}

\maketitle
\section{Introduction and preliminaries}
The topological entropy of a map $f:M\rightarrow M$, denoted by $h(f)$, measures the evolution of distinguishable orbits over time, thereby providing an idea of how complex the orbit structure of a system is. Entropy distinguishes a dynamical system where points that are close together remain close from a dynamical system where sets of points move farther. On the other hand, it measures how much $f$ mixes up the point set topology of $M$ whereas its induced homomorphism $f_*:H_*(M;\mathbb{R})\rightarrow H_*(M;\mathbb{R})$ measures how much $f$ mixes up the algebraic topology of $M$. 

In order to discover the connection between topological entropy and homology theory, there is an entropy conjecture due to Shub \cite{M. Shub} relating $h(f)$ to $r(f_*)$, the spectral radius of $f_*$. The idea is that $h(f)$ should be bounded below by $\log r(f_*)$. Unfortunately, Shub showed \cite{M. Shub} that it is not true in general for continuous maps $f$ of a manifold $M$ nor for $f$ a homeomorphism of a finite complex.

One of the known results \cite{A. Manning} is that the entropy conjecture holds for the first homology group once $M$ is a compact smooth manifold. Moreover, it is also shown in \cite{A. Manning} that $M$ does not necessarily have to be a manifold. It suffices that the space can be seen as ``locally contractible". Motivated by this relaxation, we consider building up a new homology theory such that every metric space can be regarded as locally contractible with this homology. Our initial idea is to ``rescale" the metric space. Naturally, when we observe things around us, sometimes it is hard to guarantee that everything we see is absolutely precise, especially for someone suffering from myopia. For example, given a metric space $(X,d)$,  for any $x\neq y$ with $d(x,y)=\epsilon$, they can be seen as one point when we put a ``scale" greater than $\epsilon$ onto the space. So in the second section, we introduce the definitions of comparatively ``rough" continuous maps, $\epsilon$-continuous maps, $\epsilon$-singular chain groups as well as $\epsilon$-singular homology groups $H_*^{\epsilon}(X)$. If the diameter of a space is sufficiently small, then the $\epsilon$-homology group of that space is trivial for $n>0$, i.e., it is like a point or contractible!

Since we have an explicit chain complex for $\epsilon$-homology, we can easily check that it satisfies the Eilenberg–Steenrod axioms except for homotopy axiom. Indeed, the $\epsilon$-homology is not a functor from category of metric spaces $\cM et$ to category of abelian groups $\cA b$. Concerning this drawback, we take the inverse limit of $(H_*^{\epsilon}(X))_{\epsilon\in\mathbb{R}+}$
to obtain the infinitesimal-scaled homology groups of the space and call them $lc$-homology groups where $lc$ denotes the meaning of ``local contractification". It is then a homology functor and satisfies most of the homology axioms, but it fails to satisfy exactness axiom in general. In the following we discuss two counterexamples and claim that it holds, however, for the corresponding $lc$-cohomology theory.

Our work in the last section generalizes the existing results of entropy conjecture, relaxing the restrictions on the compact metric space on which $f$ acts. The entropy conjecture will then hold for the first homology group with $lc$-homology.
\medskip

Through out this paper, all spaces are metric spaces. To begin with, let's recall the concept of topological entropy defined by Bowen \cite{R. Bowen}.

Given a compact metric space $(X,d)$ and a continuous map $f:X\rightarrow X$, we set $d_n:X\times X\rightarrow\mathbb{R}$ by 
$$
d_n(x,y)=\max_{0\le k<n}d(f^k(x),f^k(y)).
$$
The map $d_n$ is a metric on $X$ for each $n$ and we will consider the metric $d_n$ to count distinguishable orbit segments at a fixed resolution. 

\begin{definition}
Fix $\epsilon>0$ and let $n\in\mathbb{N}$. A set $S\subset X$ is said to be an $(n,\epsilon)$-spanning set for $f$ if for all $x\in X$, there exists $y\in S$ such that $d_n(x,y)<\epsilon$, and it is said to be an $(n,\epsilon)$-separated set for $f$ if for all $x\neq y\in S$, $d_n(x,y)\geq\epsilon$, i.e., there exists $k\in\{0,1,\cdots,n-1\}$ such that $d(f^k(x),f^k(y))\geq\epsilon$.
\end{definition}

Let $r_n(\epsilon,f)$ be the minimum cardinality of an $(n,\epsilon)$-spanning set and $s_n(\epsilon,f)$ be the maximum cardinality of an $(n,\epsilon)$-separated set. Then we have the following relation
$$
r_n(\epsilon,f)\le s_n(\epsilon,f)\le r_n(\frac{\epsilon}{2},f)
$$
and this implies the following equivalent definition of the topological entropy:
\begin{definition}
Let $f:X\rightarrow X$ be a continuous map. The topological entropy of $f$ is defined as
$$
h(f)=\lim\limits_{\epsilon\to 0^+}\limsup\limits_{n\rightarrow\infty}\frac{1}{n}\log
s_{n}(\epsilon,f)=\lim\limits_{\epsilon\to 0^+}\limsup\limits_{n\rightarrow\infty}\frac{1}{n}\log
r_{n}(\epsilon,f).
$$
\end{definition}

In \cite{M. Shub}, Shub stated the following conjecture.

{\bf Entropy conjecture.} Let $f$ be a continuous map on a compact finite-dimensional manifold $M$ to itself. Then 
$$
h(f)\geq \log \rho(f_*),
$$
where $f_{*}:H_{*}(M,\mathbb{R})\rightarrow H_{*}(M, \mathbb{R})$ denotes the homomorphism induced by $f$ on the total homology of $M$, i.e.,
\begin{equation}\label{expansion}
H_{*}(M,\mathbb{R})=\bigoplus\limits_{i=0}\limits^{\dim M} H_i(M,\mathbb{R})
\end{equation}
and $\rho(f_{*})$ denotes the spectral radius of $f_{*}$, i.e., $\rho(f_{*})=\lim\limits_{n\rightarrow\infty}(\|f_*^n\|^{\frac{1}{n}})$, which coincides with the maximum of the moduli of the eigenvalues of $f_*$. The entropy conjectures means that $f_{*}$ must capture some but not necessarily all of the mixing $f$ does. 

Since the expansion (\ref{expansion}) is clearly invariant under $f_*$, we have 
$$
\rho(f_*)=\max_{1\le i\le\dim M}\rho(f_{*i})
$$
where $f_{*i}$ denotes the restriction of $f_*$ to the space $H_i(M;\mathbb{R})$. Hence the entropy conjecture is equivalent to the system of inequalities
$$
h(f)\geq\log \rho(f_{*i}), \ i=1,\cdots, \dim M.
$$

In fact, the entropy conjecture only holds for some special cases. In \cite{A. Manning}, Manning showed that 
\begin{theorem}
Let $M$ be a compact differentiable manifold without boundary and $f:M\rightarrow M$ a continuous map. Then
$$
h(f)\geq\log\rho(f_{*1})
$$
where $f_{*1}:H_1(M;\mathbb{R})\rightarrow H_1(M;\mathbb{R})$ is the induced map on the first homology group.
\end{theorem}

\begin{corollary}
If $M$ has dimension $\le 3$ and $f$ is a homeomorphism, then $h(f)\geq\log \rho(f_*)$ where $f_*$ is the map induced on the homology of all dimensions.
\end{corollary}

In \cite{J.Palis}, the authors propose to imitate for homeomorphisms the ``Markov approximation" procedure for diffeomorphisms mentioned in \cite{M. Shub and D. Sullivan} and construct a dense set of homeomorphisms where the entropy conjecture holds. 
\begin{theorem}
If $\dim M\geq 5$, the entropy conjecture holds for all homeomorphisms belonging to an open dense set of the space of all homeomorphisms of $M$ with the $C^0$-topology.
\end{theorem}

Next, due to Nitecki in \cite{Z. Nitecki}, this holds for diffeomorphisms with a hyperbolic structure. Thin would imply that any sufficiently small perturbation in the $C^0$-topology of such a homeomorphism could not decrease the topological entropy. 

For arbitrary homeomorphisms by \cite{C.C. Pugh}, however, the entropy conjecture does not hold.

It is also shown in \cite{A. Manning} that it is not necessary to suppose that $M$ is a manifold. It suffices that the following conditions be met:
\begin{enumerate}
\item For any $\epsilon>0$, there exists a $\delta>0$ such that any two points $x$ and $y$ for which $d(x,y)<\delta$ can be joined by a path of diameter $<\epsilon$.
\item There exists an $\epsilon_0>0$ such that any loop of diameter $<\epsilon_0$ is contractible in $M$.
\end{enumerate}
Inspired by the relaxation of the restrictions Manning did on the space on which $f$ acts, we attempt to generalize the existing results regarding smooth manifolds to a metric scale.

\section{Scaled homology and cohomology}
\begin{definition}
Given $\epsilon>0$, a map $f:X\rightarrow Y$ is said to be $\epsilon$-continuous if there exists $\delta=\delta(\epsilon)>0$
such that 
$$
\sup\{d_Y(f(x_1), f(x_2)):d_X(x_1,x_2)<\delta\}<\epsilon.
$$
\end{definition}

Let $\Delta^n= \{(t_0,...,t_n)\in\mathbb{R}^{n+1}|\sum_{i=0}^n t_i=1\ {\rm and}\ t_i\geq 0, \forall i\} $, the standard $n$-simplex. An $\epsilon$-continuous map $\sigma:\Delta^n\rightarrow X$ is said to be an $\epsilon$-singular $n$-simplex. Moreover, we define the $n$-dimensional $\epsilon$-singular chain group, denoted by $C_n^{\epsilon}(X)$, to be the free abelian group generated by all $\epsilon$-singular $n$-simplices of $X$, and $\partial_n:C_n^{\epsilon}(X)\rightarrow C_{n-1}^{\epsilon}(X)$ to be the boundary operator defined as the one in classic singular homology theory, i.e.,
\[\partial_n\sigma=\sum_{i=0}^n(-1)^i\sigma|[v_0,\cdots,\hat{v_i},\cdots,v_n].\]
Then we call $(C_{*}^{\epsilon}(X),\partial)$ an $\epsilon$-singular chain complex.

\begin{definition}
Let $(C_*^{\epsilon}(X),\partial)$ be an $\epsilon$-singular chain complex. Define $Z_n^{\epsilon}(X)=\Ker(\partial_n)$ to be the subgroup of cycles and $B_n^{\epsilon}(X)=\Ima(\partial_{n+1})$ to be the subgroup of boundaries. Furthermore,
$H_n^{\epsilon}(X)=Z_n^{\epsilon}(X)/B_n^{\epsilon}(X)$ is said to be the  $n^{th}$ $\epsilon$-homology group of $X$.
\end{definition}
\begin{proposition} If $(X,d)$ is a metric space with diameter $<\epsilon$, {\rm i.e.}, the maximum distance between any two of its points is less than $\epsilon$, then $H_0^{\epsilon}(X)=\mathbb{Z}$ and $H_n^{\epsilon}(X)=0$ for all $n>0$. 
\end{proposition}
\begin{proof}
As in the case of singular homology theory, it is easy to see that if $X$ is a point $\{{\rm pt}\}$, then $H_n^{\epsilon}(X)=0$ for $n>0$ and $H_0^{\epsilon}(X)\cong\mathbb{Z}$ since there is a unique $\epsilon$-singular $n$-simplex $\sigma_n$ for all $n$.

Now let $f$ be a constant map sending all of $X$ to $x_0\in X$. Since ${\rm diam} X<\epsilon$, $f$ is $\epsilon$-homotopic to the identity map $\id:X\rightarrow X$, i.e., there exists $F:X\times I\rightarrow Y$ given by $F(x,t)=f_t(x)$ that is $\epsilon$-continuous such that $f_0=f$, $f_1=\id$. 

Let $\Delta^n\times {0}=[v_0,\cdots,v_n]$ and $\Delta^n\times {1}=[w_0,\cdots,w_n]$ and let $P:C_n^{\epsilon}(X)\rightarrow C_{n+1}^{\epsilon}(X)$ be such that 
$$
P(\sigma)=\sum_i(-1)^iF\cdot(\sigma\times\id)|[v_0,\cdots,v_i,w_i,\cdots,w_n]
$$
for $\sigma:\Delta^n\rightarrow X$, an $\epsilon$-continuous map. It is easy to check that $P$ is a chain homotopy between the chain maps $f_{\sharp}$ and $\id_{\sharp}$ that are induced by $f$ and $\id$ respectively, and $f_*:H_n^{\epsilon}(X)\rightarrow H_n^{\epsilon}(\{{\rm pt}\})$ are isomorphisms for all $n$. Therefore, $H_0^{\epsilon}(X)=\mathbb{Z}$ and $H_n^{\epsilon}(X)=0$ for all $n>0$.
\end{proof}

Thus, a metric space $(X,d)$ with $\epsilon$-scale can be informally seen as a ``locally contractible" space. Like in the singular homology theory, we can define the reduced $\epsilon$-homology groups $\widetilde{H}_*^{\epsilon}(X)$ to be the homology groups of the augmented chain complex
$$
\xymatrix{\cdots \ar[r] & C_{2}^{\epsilon}(X) \ar[r]^{\partial_2} & C_{1}^{\epsilon}(X) \ar[r]^{\partial_1} & C_{0}^{\epsilon}(X) \ar[r]^{\varepsilon} & \mathbb{Z} \ar[r] & 0}
$$
where $\varepsilon(\sum_ik_i\sigma_i)=\sum_ik_i$. It can be easily checked that $H_0^{\epsilon}(X)=\widetilde{H}_0^{\epsilon}(X)\oplus\mathbb{Z}$ and $H_n^{\epsilon}(X)=\widetilde{H}_n^{\epsilon}(X)$ for all $n>0$.

For a subspace $A\subset X$, let $C_n^{\epsilon}(X,A)=C_n^{\epsilon}(X)/C_n^{\epsilon}(A)$. Then $\partial:C_n^{\epsilon}(A)\rightarrow C_{n-1}^{\epsilon}(A)$ induces a quotient boundary map $\partial:C_n^{\epsilon}(X,A)\rightarrow C_{n-1}^{\epsilon}(X,A)$ and we have $\partial^2=0$. So we can define the relative $\epsilon$-homology group $H_n^{\epsilon}(X,A)=\Ker\partial/\Ima\partial$.

Thus, a metric space $(X,d)$ with $\epsilon$-scale can be informally seen as a ``locally contractible" space. Like in the singular homology theory, we can define the reduced $\epsilon$-homology groups $\widetilde{H}_*^{\epsilon}(X)$ to be the homology groups of the augmented chain complex
$$
\xymatrix{\cdots \ar[r] & C_{2}^{\epsilon}(X) \ar[r]^{\partial_2} & C_{1}^{\epsilon}(X) \ar[r]^{\partial_1} & C_{0}^{\epsilon}(X) \ar[r]^{\varepsilon} & \mathbb{Z} \ar[r] & 0}
$$
where $\varepsilon(\sum_ik_i\sigma_i)=\sum_ik_i$. It can be easily check that $H_0^{\epsilon}(X)=\widetilde{H}_0^{\epsilon}(X)\oplus\mathbb{Z}$ and $H_n^{\epsilon}(X)=\widetilde{H}_n^{\epsilon}(X)$ for all $n>0$.

For a subspace $A\subset X$, let $C_n^{\epsilon}(X,A)=C_n^{\epsilon}(X)/C_n^{\epsilon}(A)$. Then $\partial:C_n^{\epsilon}(A)\rightarrow C_{n-1}^{\epsilon}(A)$ induces a quotient boundary map $\partial:C_n^{\epsilon}(X,A)\rightarrow C_{n-1}^{\epsilon}(X,A)$ and we have $\partial^2=0$. So we can define the relative $\epsilon$-homology group $H_n^{\epsilon}(X,A)=\Ker\partial/\Ima\partial$.

An interesting object is the limit of $H_n^{\epsilon}(X)$ as $\epsilon\rightarrow 0^+$. Note that an $n$-dimensional $\mu$-singular simplex is naturally an $n$-dimensional $\epsilon$-singular simplex if $0<\mu\leq\epsilon$. Then, let $\varphi_{\mu\epsilon *}$ be the homomorphism induced by the natural inclusion chain map 
$\varphi_{\mu\epsilon}: (C_n^{\mu}(X),\partial)\hookrightarrow (C_n^{\epsilon}(X),\partial)$.
In particular, $\varphi_{\epsilon\epsilon *}$ is the identity on $H_n^{\epsilon}(X)$.

Since $(\mathbb{R}^+, \leq)$ is a directed partially-ordered set, we obtain an inverse system 
$$
((H_n^{\epsilon}(X))_{\epsilon\in\mathbb{R}^+}, (\varphi_{\mu\epsilon *})_{\mu\leq\epsilon\in\mathbb{R}^+}).
$$
Furthermore, we can take the inverse limit
\[\begin{split}
 H_n^{lc}(X)&=\lim\limits_{\overleftarrow{\epsilon\in\mathbb{R}^+}}((H_n^{\epsilon}(X))_{\epsilon\in\mathbb{R}^+}, (\varphi_{\mu\epsilon *})_{\mu\leq\epsilon\in\mathbb{R}^+})\\&=\{ \overrightarrow{[a]}\in\prod\limits_{\epsilon\in\mathbb{R}^+}H_n^{\epsilon}(X): [a_{\epsilon}]=\varphi_{\mu\epsilon *}([a_{\mu}]) {\rm \ for \ all} \
\mu\leq\epsilon \ {\rm in} \ \mathbb{R}^+\}.   
\end{split}\]
Let $\{\epsilon_m\}$ be a non-increasing sequence of $\mathbb{R}^{+}$ converging to $0$.
Then we have
\[
\lim\limits_{\overleftarrow{\epsilon\in\mathbb{R}^+}}((H_n^{\epsilon}(X))_{\epsilon\in\mathbb{R}^+}, (\varphi_{\mu\epsilon *})_{\mu\leq\epsilon\in\mathbb{R}^+})\cong \varprojlim((H_n^{\epsilon_m}(X))_{m=1}^{\infty}, (\varphi_{\epsilon_k\epsilon_l *})_{l\leq k\in\mathbb{N}}),\]
since $\{\epsilon_m\}$ is a cofinal subset of the directed index set $\mathbb{R}^+$. We call the inverse limit above the $n^{th}$ infinitesimal-scaled homology group of $X$ or $lc$-homology group of $X$ (the homology group under infinitesimal scale or the homology group of the ``local contractification" of $X$).  

The $lc$-homology group of $X$ is said to be stable if there exists $\epsilon>0$ such that for all $0<\epsilon_2\leq\epsilon_1\leq\epsilon$, $\varphi_{\epsilon_2,\epsilon_1 *}$ is an isomorphism, i.e., for any $0<\mu\leq\epsilon$, $H_n^{\mu}(X)\cong H_n^{\epsilon}(X)$. Later we will show that when $X$ is a Riemannian manifold, its $lc$-homology group will be stable.


\medskip
To begin with, we check that $lc$-homology is a functor from category of topological spaces $\cM et^u$ to category of abelian groups $\cA b$, where $\cM et^u$ denotes the category that has metric spaces as its objects and uniformly continous maps between metric spaces as its morphisms.

\begin{proposition}\label{homlc}
If $f:X\rightarrow Y$ is a uniformly continuous map, then $f$ induces a homomorphism $f^{lc}_{*}: H_{n}^{lc}(X)\rightarrow H_{n}^{lc}(Y),  \forall n\geq 0$.
\end{proposition}
\begin{proof}
Let $\sigma:\Delta^n\rightarrow X$ be an $\epsilon$-continuous map. Composing with $f$, we get a $\mu$-continuous map $f_{\sharp}^{\epsilon\mu}(\sigma)=f\sigma: \Delta^n\rightarrow Y$, where $\mu$ is determined by $f$ and $\sigma$. In fact, we can choose $\mu=\inf\{r>0:{\rm whenever} \ d_X(\sigma(v_1), \sigma(v_2))<\epsilon, d_Y(f(\sigma(v_1)), f(\sigma(v_2)))<r\}$. Extend $f_\sharp^{\epsilon\mu}: C_n^\epsilon(X)\rightarrow C_n^{\mu}(Y)$ linearly by $f_\sharp^{\epsilon\mu}(\sum_in_i\sigma_i)=\sum_in_if_\sharp^{\epsilon\mu}(\sigma_i)$ and we have $f_\sharp^{\epsilon\mu}\partial=\partial f_\sharp^{\epsilon\mu}$, which implies that $f_\sharp^{\epsilon\mu}$ takes cycles to cycles and takes boundaries to boundaries. Hence $f_\sharp$ induces a homomorphism $f_*^{\epsilon\mu}:H_n^\epsilon(X)\rightarrow H_n^\mu(Y)$. 

Let $\{\epsilon_n\}$ be a monotonically non-increasing sequence such that $\lim_{n\to\infty}\epsilon_n=0$. For each $\epsilon_k$-continuous map  $\varphi:\Delta^n\rightarrow X$, composing with $f$, we get a corresponding $\mu_k$-continuous map $f_\sharp^{\epsilon_k\mu_k}(\varphi)=f\varphi:\Delta^n\rightarrow Y$, where $\mu_k=\inf\{r_k>0:{\rm if} \ d_X(\sigma(v_1), \sigma(v_2))<\epsilon_k, d_Y(f(\sigma(v_1)), f(\sigma(v_2)))<r_k\}$. Since $f$ is uniformly continuous, $\mu_n\rightarrow 0$ as $\epsilon_n\rightarrow 0$. Hence, after taking the inverse limit, we obtain the induced homomorphism $f^{lc}_{*}: H_{n}^{lc}(X)\rightarrow H_{n}^{lc}(Y)$.
\end{proof}

Naturally, we can define the corresponding $lc$-cohomology group of $X$ in the following ways.

Let $G$ be an abelian group. For the $\epsilon$-singular chain complexes
$$
\xymatrix{\cdots \ar[r] & C_{n+1}^{\epsilon}(X) \ar[r]^{\partial} & C_{n}^{\epsilon}(X) \ar[r]^{\partial} & C_{n-1}^{\epsilon}(X) \ar[r] & \cdots}
$$
we apply the $\Hom(-,G)$ functor and obtain the corresponding $\epsilon$-singular cochain complex
$$
\xymatrix{\cdots & C^{n+1}_{\epsilon}(X;G) \ar[l] & C^{n}_{\epsilon}(X;G)\ar[l]_{\delta} & C^{n-1}_{\epsilon}(X;G)\ar[l]_{\delta} & \cdots\ar[l]}
$$
where $C_{\epsilon}^n(X;G)=\Hom(C_n^{\epsilon}(X);G)$ and $\delta=\partial^*:C_{\epsilon}^n(X;G)\rightarrow C_{\epsilon}^{n+1}(X;G)$ sending $\varphi\in C_{\epsilon}^n(X;G)$ to $\varphi\partial\in C_{\epsilon}^{n+1}(X;G)$. Then we can define the $\epsilon$-singular cohomology group 
$$
H_{\epsilon}^n(X;G)=\Ker\delta/\Ima\delta.
$$
Let $i_{\mu\epsilon}^*:H_{\epsilon}^n(X;G)\rightarrow H_{\mu}^n(X;G)$ be the homomorphism induced by the natural inclusion $i_{\mu\epsilon}:C_n^{\mu}(X)\hookrightarrow C_n^{\epsilon}(X)$ and we obtain a direct system $((H_{\epsilon}^n(X))_{\epsilon\in\mathbb{R}^+},(i_{\mu\epsilon}^*)_{\mu\le\epsilon\in\mathbb{R}^+})$. Then we can take the direct limit
\[
H_{lc}^n(X)=\lim\limits_{\overrightarrow{\epsilon\in\mathbb{R}^+}}((H_{\epsilon}^n(X))_{\epsilon\in\mathbb{R}^+}, (i_{\mu\epsilon }^*)_{\mu\leq\epsilon\in\mathbb{R}^+})=\bigoplus\limits_{\epsilon\in\mathbb{R}^+} H_{\epsilon}^n(X)/S
\]
where $S$ is generated by
$$\{q_{\mu}i_{\mu\epsilon}^*([\varphi_\epsilon])-q_{\epsilon}[\varphi_{\epsilon}]:[\varphi_{\epsilon}]\in H_{\epsilon}^n(X), q_{\epsilon} \ {\rm is \ the \ embedding}: H_{\epsilon}^n(X)\hookrightarrow\bigoplus\limits_{\epsilon\in\mathbb{R}^+} H_{\epsilon}^n(X)\}.$$

Next, we will check whether the $lc$-homology satisfies the Eilenberg-Steenrod axioms of homology theory. Obviously, the $lc$-homology theory satisfies the dimension axiom.
\begin{theorem}[Dimension Axiom]\label{Dimension axiom}
$H_n^{lc}(\{pt\})=0$ for all $n>0$.
\end{theorem}

\begin{theorem}[Homotopy Axiom]\label{Homotopy axiom}
Let $f,g:X\rightarrow Y$ be two uniformly continuous maps. If $f$ and $g$ are uniformly homotopic, {\rm i.e.}, there is a family of uniformly continuous maps $f_t: X\rightarrow Y, t\in [0,1]$ such that $f_0=f$, $g_0=g$ and the associated map $F:X\times [0,1]\rightarrow Y$ given by $F(x,t)=f_t(x)$ is uniformly continuous, then $f^{lc}_{*}=g^{lc}_{*}: H_{n}^{lc}(X)\rightarrow H_{n}^{lc}(Y), \forall n\geq 0$.
\end{theorem}
\begin{proof}
Let $\{\epsilon_n\}$ be a non-increasing sequence of $\mathbb{R}^{+}$ with $\lim_{n\to\infty}\epsilon_n=0.$ and we can take 
\[\begin{split}
 H_n^{lc}(X)&=\varprojlim((H_n^{\epsilon_m}(X))_{m=1}^{\infty}, (\varphi_{\epsilon_k,\epsilon_l *})_{l\leq k\in\mathbb{N}})\\&=\{ \overrightarrow{[a]}\in\prod\limits_{m=1}^{\infty}H_n^{\epsilon_m}(X): [a_{\epsilon_l}]=\varphi_{\epsilon_k,\epsilon_l*}([a_{\epsilon_k}]) {\rm \ for \ all} \
l\leq k\in\mathbb{N}\}. 
\end{split}\]
Then, we have 
\[\begin{split}
f_*^{lc}(\overrightarrow{[a]})=\{[\overrightarrow{f(a)}]:[f(a_{\epsilon_l})]=\varphi_{\mu_k,\mu_l*}([f(a_{\epsilon_k})]) {\rm \ for \ all} \
l\leq k\in\mathbb{N}\} \ {\rm and} \\
g_*^{lc}(\overrightarrow{[a]})=\{[\overrightarrow{g(a)}]:[g(a_{\epsilon_l})]=\varphi_{\omega_k,\omega_l*}([g(a_{\epsilon_k})]) {\rm \ for \ all} \
l\leq k\in\mathbb{N}\},
\end{split}\]
where an $\epsilon_*$-continuous map composed with $f$ and $g$ will be a $\mu_*$-continuous map and an $\omega_*$-continuous map respectively. Since a $\mu_k$-continuous map is naturally an $\omega_k$-continuous map if $\mu_k\leq\omega_k$, without loss of generality, we can assume $\mu_*=\omega_*$. 

Next, given a homotopy $F:X\times I\rightarrow Y$ from $f$ to $g$ with associated maps uniformly continuous and an $\epsilon_k$-singular simplex $\sigma:\Delta^n\rightarrow X$, we can form the composition $F\cdot(\sigma\times \mathrm{id}):\Delta^n\times I\rightarrow X\times I\rightarrow Y$ with associated maps $\theta_k$-continuous. As in the case of singular homology theory, we can construct a \emph{prism operator} which cuts $\Delta^n\times I$ into a sum of $(n+1)$-simplices in $\Delta^n\times I$ by
\[P(\sigma)=\sum_{i=0}^n(-1)^iF\cdot(\sigma\times\mathrm{id})|_{[v_0,\cdots,v_i,w_i,\cdots,w_n]}\]
where $\Delta^n\times {0}=[v_0,\cdots,v_n]$ and $\Delta^n\times {1}=[w_0,\cdots,w_n]$. Then $P(\sigma)\in C_{n+1}^{\theta_k}(Y)$. Pick $\tau_k=\max\{\mu_k,\theta_k\}$ and we have $P(\sigma), g_\sharp-f_\sharp-P\partial\in C_{n+1}^{\tau_k}(Y)$. As we did in singular homology theory, 
\[
\partial P(\sigma)=g_\sharp(\sigma)-f_\sharp(\sigma)-P\partial(\sigma).\]
Hence, $f$ and $g$ induce the same homomorphism $f_*=g_*: H_n^{\epsilon_k}(X)\rightarrow H_n^{\tau_k}(Y), k=1,2,\cdots$.
By taking the inverse limit, $f^{lc}_{*}=g^{lc}_{*}: H_{n}^{lc}(X)\rightarrow H_{n}^{lc}(Y), \forall n\geq 0$.
\end{proof}

\begin{corollary} The maps $f_*^{lc}:H_n^{lc}(X)\rightarrow H_n^{lc}(Y)$ induced by a uniform homotopy equivalence $f:X\rightarrow Y$ are isomorphisms for all $n$. In particular, if X is compact and contractible, then $\widetilde{H}_n^{lc}(X)=0$ for any $n$.
\end{corollary}
\begin{remark}
Actually, we cannot simply omit the restriction "uniformly homotopic". Indeed, let $p\in S^1$ and consider a space $S^1\setminus \{p\}$. If $\epsilon$ is sufficiently small, then we cannot distinguish between $C_n(S^1)$ and $C_n^{\epsilon}(S^1\setminus\{p\})$. Naturally, we have $H_n(S^1)\cong H_n^{\epsilon}(S^1\setminus\{p\})$ (we will give a rigorous proof later in this section) since $B(p,\frac{1}{2}\epsilon)\setminus\{p\}$ is contractible under the "$\epsilon$-scale" and can be regarded as a point. By taking inverse limit, we have $H_1^{lc}(S^1\setminus\{p\})=\mathbb{Z}$, even though $S^1\setminus\{p\}$ is obviously contractible. The reason for this is that the homotopy map connecting the retraction of $S^1\setminus\{p\}$ to a point and the identity map is not uniformly continuous.
\end{remark}

\begin{theorem}[Excision Axiom]\label{Excision axiom}
Let $Z$ and $A$ be two subspaces of $X$ such that the closure of $Z$ is contained in the interior of $A$, {\rm i.e.}, $\bar{Z}\subseteq \mathring{A}$ and ${\rm dist}(\bar{Z},\partial A)>0$. Then the inclusion map $i: (X-Z, A-Z)\rightarrow (X,A)$ induces an isomorphism $i^{lc}_{*}: H_{n}^{lc}(X-Z, A-Z)\rightarrow H_{n}^{lc}(X,A)$ for all $n$. Equivalently, for subspaces $A,B\subseteq X$ with $X=\mathring{A}\cup\mathring{B}$, the inclusion $(B,A\cap B)\hookrightarrow (X,A)$ induces isomorphisms $H_n^{lc}(B,A\cap B)\rightarrow H_n^{lc}(X, A)$ for all $n$.
\end{theorem}
\begin{proof}
For the cover $\{A,B\}$ with $X=A\cup B$ and $Z=X-B$, pick $0<\epsilon<{\rm dist}(\bar{Z},\partial A)$ and denote $C_n^{\epsilon\{A,B\}}(X)$ by $C_n^{\epsilon}(A+B)$, the formal sums of chains in $A$ and chains in $B$ (here $\epsilon$ is globally fixed for this space). As we do for singular homology theory, by barycentric subdivision of chains, we can construct a chain map $S:C_n^{\epsilon}(X)\rightarrow C_n^{\epsilon}(X)$ and a chain homotopy $T:C_n^{\epsilon}(X)\rightarrow C_{n+1}^{\epsilon}(X)$, s.t.
$$
\partial T+T\partial={\rm id}-S.
$$
In addition, for the cover ${A,B}$ of $X$ and any $\epsilon$-simplex $\sigma:\Delta^n\rightarrow X$, there exists $m\in\mathbb{N}$, s.t. $S^m\sigma\in C_n^{\epsilon\{A,B\}}(X)$ since $\dist (A\cap B)>\epsilon$. It will follow that $S$, $T$ restrict to maps on $C_n^{\epsilon\{A,B\}}(X)$. Let $D_m=\sum_{0\leq i<m}TS^i$ and it satisfies 
$$\partial D_m+D_m\partial={\rm id}-S^m.$$
Hence the iterate $S^m$ is chain homotopic to ${\rm id}$. 

Then the inclusion $C_n^{\epsilon}(A+B)/C_n^{\epsilon}(A)\hookrightarrow C_n^{\epsilon}(X)/C_n^{\epsilon}(A)$ induces an isomorphism on homology. In addition, we have a natural isomorphism $C_n^{\epsilon}(B)/C_n^{\epsilon}(A \cap B)\rightarrow C_n^{\epsilon}(A+B)/C_n^{\epsilon}(A)$ since both sides are free abelian groups generated by the $\epsilon$-singular $n$-simplices in $B$ that are not contained in $A$. Composing these two isomorphisms on homology, we get $H_n^{\epsilon}(B,A\cap B)\cong H_n^{\epsilon}(X, A)$ for all $n$. By taking the inverse limit, $H_n^{lc}(B,A\cap B)\cong H_n^{lc}(X, A)$.
\end{proof}

As a corollary of excision theorem, we naturally have:
\begin{theorem}[Refined Additivity Axiom]\label{Additivity axiom}
Let $X=\bigsqcup\limits_{\lambda\in\Lambda}X_{\lambda}$. If there exists $\epsilon>0$ such that for any $k,\lambda\in\Lambda$, $\dist (X_{k},X_{\lambda})>\epsilon$, then
$H_n^{lc}(X)\cong \bigoplus\limits_{\lambda\in\Lambda}H_n^{lc}(X_{\lambda})$ for all $n$.
\end{theorem}

\begin{example} $X=\{\displaystyle\frac{1}{2^k};k=0,1,2,\cdots\}\cup\{0\}$.

Take $\epsilon_k=1/2^{k+1}$. For a fixed $\epsilon_k$, $H_0^{\epsilon_k}(X)$ is the free abelian group generated by elements of $X$ that are $\geq 1/2^{k+1}$ and $\{0\}$, i.e., elements of $H_0^{\epsilon_k}(X)$ are given by $(a_0,a_1,\cdots,a_{k+2})\in\mathbb{Z}^{k+3}$. Hence the inverse limit of $(H_0^{\epsilon_k}(X))_{n\in\mathbb{N}}$ is $H_0^{lc}(X)=\mathbb{Z}^{|X|}$, the direct product of $\mathbb{Z}$, where $|X|$ denotes the cardinality of $X$. 

Next, for $\epsilon_k$, $X$ can be written as $X=X_1^{(k)}\bigsqcup X_2^{(k)}$ with 
$$
X_1^{(k)}=\{0,1,\frac{1}{2},\cdots,\frac{1}{2^k}\},X_2^{(k)}=\{\frac{1}{2^{k+1}},\frac{1}{2^{k+2}},
\cdots\}
$$
and for $n\geq 1$, $H_n^{\epsilon_k}(X)\cong H_n^{\epsilon_k}(X_1^{(k)})\oplus H_n^{\epsilon_k}(X_2^{(k)})$. By dimension axiom, 
$$
H_n^{\epsilon_k}(X_n^{(k)})\cong  H_n^{\epsilon_k}(\{0\})\oplus H_n^{\epsilon_k}(\{1\})\oplus\cdots\oplus H_n^{\epsilon_k}(\{\frac{1}{2^k}\})=0.
$$
Since the diameter of $X_2^{(k)}$ is less than $\epsilon_k$, by proposition 2.3, $H_n^{\epsilon_k}(X_2^{(k)})=0$. Therefore, $H_n^{\epsilon_k}(X)=0$ holds for all $\epsilon_k$ and by taking the inverse limit, $H_n^{lc}(X)=0$ for $n\geq 1$.

\end{example}

Consider the commutative diagram
$$
\xymatrix{0\ar[r] & C_n^{\epsilon}(A)\ar[r]^{i_{\epsilon}} \ar[d]^{\partial} & C_n^{\epsilon}(X) \ar[r]^{j_{\epsilon}}\ar[d]^{\partial} & C_n^{\epsilon}(X,A)\ar[r]\ar[d]^{\partial} & 0\\
0\ar[r] & C_{n-1}^{\epsilon}(A)\ar[r]_{i_{\epsilon}} & C_{n-1}^{\epsilon}(X)\ar[r]_{j_{\epsilon}} & C_{n-1}^{\epsilon}(X,A)\ar[r] & 0
}
$$
where $i_{\epsilon}$ is inclusion and $j_{\epsilon}$ is the quotient map. By the Snake Lemma, we obtain an induced long exact sequence in $\epsilon$-homology:
$$
\xymatrix{\cdots \ar[r] & H_{n}^{\epsilon}(A) \ar[r]^{i_{\epsilon*}} & H_{n}^{\epsilon}(X) \ar[r]^{j_{\epsilon*}} & H_{n}^{\epsilon}(X,A) \ar[r]^{\partial_{\epsilon}} & H_{n-1}^{\epsilon}(A) \ar[r] & \cdots}
$$
Then we have the following proposition.
\begin{proposition}
Let $A$ be a subspace of $X$. Then we have $H_n^{lc}(A)\cong H_n^{lc}(\bar{A})$ for all $n$.
\end{proposition}
\begin{proof}
First we claim that for any $\epsilon>0$, $H_n^{\epsilon}(A)\cong H_n^{\epsilon}(\bar{A})$. Indeed, by the long exact sequence in $\epsilon$-homology
$$
\xymatrix{\cdots \ar[r] & H_{n+1}^{\epsilon}(\bar{A},A) \ar[r]^{\partial_{\epsilon}} & H_{n}^{\epsilon}(A) \ar[r]^{i_{\epsilon*}} & H_{n}^{\epsilon}(\bar{A}) \ar[r]^{j_{\epsilon*}} & H_{n}^{\epsilon}(\bar{A},A) \ar[r] & \cdots}
$$
it suffices to show that $H_n^{\epsilon}(\bar{A},A)=0$ for all $n$.

Let $[\alpha]\in H_n^{\epsilon}(\bar{A},A)$ where $\alpha\in C_n^{\epsilon}(\bar{A})$ and $\partial\alpha\in C_{n-1}^{\epsilon}(A)$. Write $\alpha=\sum_{i=1}^k n_i\alpha_i$ where $\alpha_i:\Delta^n\rightarrow\bar{A}$ is $\epsilon$-continuous. For each $\alpha_i$, $\exists \delta_i>0$, s.t. $\omega_i=\sup\{d(\alpha_i(x_1), \alpha_i(x_2)):d(x_1,x_2)<\delta_i\}<\epsilon$. By definition, 
$$
\partial\alpha=\sum_{i=1}^k\sum_{j=0}^nn_i\alpha_i|[v_0,\cdots,\hat{v_j},\cdots,v_n]
$$
which is a formal sum of $\epsilon$-continuous maps. Hence the restrictions of $\alpha_i:\Delta^n\rightarrow\bar{A}$ to the faces of $\Delta^n$ consist of $\epsilon$-continuous maps whose images are contained in either $\bar{A}$ or $A$, and in the formal sum of $\partial\alpha$, all of those $\epsilon$-continuous maps whose images intersect $\partial A$ are cancelled. 

Then, we construct $\beta_i:\Delta^n\times I\rightarrow\bar{A}$ by 
$$
\beta_i(\sum_{\substack{l=0\\l\neq j}}^n\lambda_lv_l,t)\equiv\alpha_i(\sum_{\substack{l=0\\l\neq j}}^n\lambda_lv_l)
$$
where $\sum_{\substack{l=0\\l\neq j}}^n\lambda_l=1$ for $j=0,1,\cdots,n$ and 
$$
\beta_i(\sum_{l=0}^n\lambda_lv_l,0)=\alpha_i(\sum_{l=0}^n\lambda_lv_l)
$$
where $\sum_{l=0}^n\lambda_l=1$. 

Moreover, for other points $(x,t)\in\Delta^n\times I$, if $\alpha_i(x)\in A$, then we let $\beta_i(x,t)=\alpha_i(x)$. On the other hand, if $\alpha_i(x)\notin A$, we do a perturbation for $\alpha_i(x)$ under the scale of $(\epsilon-\omega_i)/3$ and we obtain $\widetilde{\alpha_i}(x)$. Then we let $\beta_i(x,t)=\widetilde{\alpha_i}(x)$. Clearly, $\beta_i$ is still $\epsilon$-continuous and since $\Delta^n\times I$ is a simplicial complex, $\beta_i$ can be seen as a formal sum of $\epsilon$-singular $n+1$-simplices. Let $\beta=\sum_{i=1}^kn_i\beta_i\in C_{n+1}^{\epsilon}(\bar{A})$. Then $\del\beta-\alpha\in C_n^{\epsilon}(A)$. So $[\alpha]=0$ in $H_n^{\epsilon}(\bar{A},A)$ and $H_n^{\epsilon}(A)\cong H_n^{\epsilon}(\bar{A})$ for $\forall \epsilon>0$.

Finally, by taking the inverse limit, $H_n^{lc}(A)\cong H_n^{lc}(\bar{A})$ for all $n$.
\end{proof}

But unfortunately, the inverse limit functor preserving left exactness is still not exact in our case.

\begin{example}
Let $A$ be the set of positive integers and let $X$ be the set of all $n+\frac in$ where $i,n$ are integers with $0\le i< n$, i.e.,
\[
	X=\{1,\  2, 2\tfrac12,\ 3, 3\tfrac13,3\tfrac23,\ 4,4\tfrac14,4\tfrac12,4\tfrac34,\ 5,5\tfrac15,\cdots\}
\]

Consider the short exact sequence 
$$
\xymatrix{0 \ar[r] & \Ker i_{\epsilon*} \ar[r] & H_{0}^{\epsilon}(A) \ar[r] & \Ima i_{\epsilon*} \ar[r]& 0}
$$
where $H_0^{\epsilon}(A)$ is the same for all $0<\epsilon<1$, which is the direct sum of countably infinitely many copies of $\mathbb{Z}$ and $H_0^{\epsilon}(X)$ is a finitely generated abelian group. Without loss of generality, take $\epsilon_n=1/n$. Then elements of $i_{\epsilon_n*}(H_0^{\epsilon_n}(A))\subset H_0^{\epsilon_n}(X)$ are given by a sequence of integers $(a_0,a_1,\cdots,a_n)$, i.e., $i_{\epsilon_n*}(H_0^{\epsilon_n}(A))$ is generated by only $n$ generators.

Let $([a]_{\epsilon_n})\in\varprojlim\Ima(i_{\epsilon_n})$, which can be written as
\[
((a_1,a_2),(a_1,a_2,a_3),\cdots,(a_1,a_2,\cdots, a_m),\cdots)
\]
and can have countably infinite length (the number of nonzero coefficients). 
But the element of $\varprojlim H_0^{\epsilon_n}(A)$ can be characterized by the element of the direct sum of countably infinitely many copies of $\mathbb{Z}$, so its length can only be finite. Therefore, there is no element of $H_0^{lc}(A)$ that can go to the elements with infinitely many nonzero coordinates in $\varprojlim\Ima i_{\epsilon_n*}$, i.e., $H_0^{lc}(A)\rightarrow\varprojlim\Ima i_{\epsilon_n*}$ is not surjective and we fail to show the exactness at $H_0^{\epsilon}(X)$. 
\end{example}

Furthermore, we can give a counterexample showing that the $lc$-homology doesn't satisfy the exactness axiom.

\begin{example}
Let $B$ be the set of all $n+\frac in$ where $i,n$ are  integers with $0\le i< n$, i.e.,
\[
	B=\{1,\  2, 2\tfrac12,\ 3, 3\tfrac13,3\tfrac23,\ 4,4\tfrac14,4\tfrac12,4\tfrac34,\ 5,5\tfrac15,\cdots\}.
\]
Take $\epsilon_n$ to be such that $1/n<\epsilon_n\le 1/(n-1)$. Clearly, $H_0^{\epsilon_n}(B)$ is the free abelian group generated by the elements of $B$ which are $\le n$. For instance, for $n=3$, $H_0^{\epsilon_3}(B)=\mathbb Z^4$, which is freely generated by $\{[1], [2], [2.5], [3]=[3\frac13]=[3\frac23]\}$. Thus, elements of $H_0^{\epsilon_3}(B)$ are given by the set of tuples $\{(a_1,a_2, a_{5/2},a_3)\in \mathbb Z^4\}$. Add a ``dummy variable'' $a_0=-(a_1+a_2+ a_{5/2}+a_3)$ and we obtain a sequence of five integers $(a_0,a_1,a_2, a_{5/2},a_3)$ whose sum is $0$. 

Take $B_+=B\coprod \{0\}$. Then the inverse limit of $(H_0^{\epsilon_n})_{n\in\mathbb{N}}$ is $H_0^{lc}(B)=\mathbb Z^{|B_+|}$, the set of all infinite sequences of integers $(a_0,a_1,a_2, a_{5/2},\dots)$ indexed by the elements of $B_+$.

Let $A=B\times \{0,1\}=B\coprod B$, the disjoint union of two copies of $B$ with distance $>\epsilon_n$. Then $H_0^{lc}(A)=H_0^{lc}(B)\oplus H_0^{lc}(B)$ whose elements are given by pairs of infinite sequences of integers
$$\left((a_0,a_1,a_2, a_{5/2},\dots),(b_0,b_1,b_2, b_{5/2},\dots)
\right).$$

Take $X=A\cup\{1,2,\cdots\}\times I$. For any $\epsilon<1$, $H_1^\epsilon(X,A)$ is the free abelian group generated by the set of vertical line segments in $X$, which is the direct sum of countable copies of $\mathbb{Z}$. Therefore, $H_1^{\epsilon}(X,A)\cong H_1^{lc}(X,A)$ and elements in $H_1^{lc}(X,A)$ can only have finite nonzero coefficients. 

Note that the homomorphism $H_0^{lc}(A)\to H_0^{lc}(X)$ adds together the coefficients $a_n,b_n$ for all integers $n$, so the kernel consists of all pairs of infinite sequence $(a_\ast,b_\ast)$ such that $a_k+b_k=0$ for all integers $k$ and $a_x=b_x=0$ for all $x\in B$ that are not integers. So the kernel of $H_0^{lc}(A)\rightarrow H_0^{lc}(X)$ is a product of infinitely many copies of $\mathbb Z$, i.e., it can have countably infinite nonzero coefficients that cannot be the image of $H_1^{lc}(X,A)$. 

Therefore, the sequence $$\xymatrix{\cdots\ar[r]& H_1^{lc}(X,A) \ar[r] & H_0^{lc}(A) \ar[r] & H_0^{lc}(X)}
$$ 
is not exact at $H_0^{lc}(A)$.
\end{example}

\begin{remark}
However, $lc$-cohomology satisfies the exactness axiom. Indeed, the $lc$-cohomology sequence
$$
\xymatrix {\cdots\ar[r] & H^n_{lc}(X) \ar[r] & H^n_{lc}(A) \ar[r] & H^{n+1}_{lc}(X,A) \ar[r]& H^{n+1}_{lc}(X)\ar[r]&\cdots}
$$
with any coefficients is exact since it is the direct limit of the exact sequences
$$
\xymatrix {\cdots\ar[r] & H^n_{\epsilon}(X) \ar[r] & H^n_{\epsilon}(A) \ar[r] & H^{n+1}_{\epsilon}(X,A) \ar[r]& H^{n+1}_{\epsilon}(X)\ar[r]&\cdots}
$$
and the direct limit serves as an exact functor.

In particular, for our counterexample, the $lc$-cohomologies are all countably generated since they are countable direct limits of countably generated groups.

Therefore, the Universal Coefficient Theorem does not hold for $lc$-cohomology theory in general. 
\end{remark}

To end this section, we will show the equivalence of $lc$-homology and singular homology in one special case.
\begin{theorem}\label{Equivalence}
Let $(X,d)$ be a compact $n$-dimensional manifold endowed with a Riemannian structure. Then the homomorphisms $H_n(X)\rightarrow H_n^{lc}(X)$ are isomorphisms for all $n$.
\end{theorem} 
\begin{proof}
Since $X$ is a Riemannian manifold, by Theorem 5.1 in \cite{Bott}, every point has a geodesically convex neighborhood. So for any $x\in X$, there exists $\epsilon_x>0$ such that $B(x,\epsilon_x)\cong\mathbb{R}^n$ and $X=\bigcup_{x\in X}B(x,\epsilon_x)$. By compactnesss, there exists $N>0$ such that $X$ can be expressed as $\bigcup_{i=1}^NB(x_i,\epsilon_{x_i})$. Let $\delta>0$ be a Lebesgue number for this cover, i.e., for $\forall x\in X$, $B(x,\delta)\subset B(x_k,\epsilon_{x_k})$ for some $1\leq k\leq N$. Hence each point of $X$ has a $\delta$-ball homeomorphic to $\mathbb{R}^n$.

Now for $\mu\leq\epsilon<\delta$, consider the short exact sequence
$$
\xymatrix{0 \ar[r] & C_n^{\mu}(X)\ar[r]& C_n^{\epsilon}(X) \ar[r] & C_n^{\epsilon}(X)/C_n^{\mu}(X)\ar[r]&0}
$$
and by Snake Lemma, we have the following long exact sequence
$$
\xymatrix{\cdots \ar[r] & H_{n+1}(X^{\epsilon},X^{\mu}) \ar[r]& H_{n}^{\mu}(X) \ar[r] & H_{n}^{\epsilon}(X) \ar[r] & H_n(X^{\epsilon},X^{\mu}) \ar[r] & \cdots}
$$
where $H_n(X^{\epsilon},X^{\mu})$ denotes the homology group of
$$
\partial:C_n^{\epsilon}(X)/C_n^{\mu}(X)\rightarrow C_{n-1}^{\epsilon}(X)/C_{n-1}^{\mu}(X).
$$
Then we claim that $H_n(X^{\epsilon},X^{\mu})=0$ for all $n$.

Let $[\alpha]\in H_n(X^{\epsilon},X^{\mu})$ where $\alpha\in C_n^{\epsilon}(X)$ and $\partial\alpha\in C_{n-1}^{\mu}(X)$. Write $\alpha=\sum_{i=1}^kn_i\sigma_i$ with $\sigma_i:\Delta^n\rightarrow X$ an $\epsilon$-continuous map. By definition,
$$
\partial\alpha=\sum_{i=1}^k\sum_{j=0}^nn_i\sigma_i|[v_0,\cdots,\hat{v_j},\cdots,v_n]
$$
which is a formal sum of $\mu$-continuous maps. Hence the restrictions of $\sigma_i:\Delta^n\rightarrow X$ to the faces of $\Delta^n$ consist of $\mu$-continuous maps and $\epsilon$-continuous maps and in the formal sum of $\partial\alpha$, all of those $\epsilon$-continuous maps are cancelled. 

For $\sigma_i$, for any $x\in\Ima(\sigma_i)\subset X$, $\bigcup_{x\in\Ima(\sigma_i)}B(x,\delta)$ forms an open cover of its image in $X$. Define $S:C_n^{\epsilon}(X)\rightarrow C_n^{\epsilon}(X)$ by sending each $\epsilon$-singular $n$-simplex $\sigma:\Delta^n\rightarrow X$ to $\sigma_{\sharp}S\Delta^n$ where $S\Delta^n$ is the signed sum of the $n$-simplices in the barycentric subdivision of $\Delta^n$ and $S\sigma$ is the corresponding signed sum of the restrictions of $\sigma$ to the $n$-simplices of the barycentric subdivision of $\Delta^n$. Since $\mu\leq\epsilon<\delta$, we can guarantee that there exists $m>0$, s.t. $S^m\sigma_i=\sum_{l=1}^Nk_l\sigma_{i_l}$ with $k_l=\pm1$, $\Ima(\sigma_{i_l})\subset B(x_l,\delta)$ for some $x_l\in\Ima(\sigma_i)$, where $\sigma_{i_l}:\Delta_l^n\rightarrow X$ and $\bigcup_{l=1}^N\Delta_l^n=\Delta^n$. 

Next, for each $\sigma_{i_l}:[v_0,v_1,\cdots,v_n]\rightarrow X$, let $\varphi_l:B(x_l,\delta)\rightarrow\mathbb{R}^n$ be the coordinate map and let $f_{il}=\varphi_l\circ\sigma_{il}:\Delta_l^n\rightarrow\mathbb{R}^n$. Then we do the following construction,
\begin{align*}
F_{il}:\Delta_l^n\times I&\rightarrow\mathbb{R}^n\\
(\sum_{j=0}^n\lambda_jv_j,t)&\mapsto (1-t)f_l(\sum_{j=0}^n\lambda_jv_j)+t\sum_{j=0}^n\lambda_jf_l(v_j).
\end{align*}

\begin{center}
\begin{tikzpicture}\tiny
\draw[fill,blue!7!white](4,0)--(6.5,0)--(5.5,1)--cycle;
\draw[color=blue,thick](4,0)--(6.5,0)--(5.5,1)--cycle;
\draw[dashed,thick](4,-2.5)--(5.5,-1.5)--(6.5,-2.5);
\draw[thick](4,-2.5)--(6.5,-2.5);
\draw[thick](4,0)--(4,-2.5);
\draw[thick](6.5,0)--(6.5,-2.5);
\draw[dashed,thick](5.5,1)--(5.5,-1.5);
\filldraw[blue](4,0) circle(1.5pt) node[left]{$(v_0,1)$};
\filldraw[blue](6.5,0) circle(1.5pt) node[right]{$(v_1,1)$};
\filldraw[blue](5.5,1) circle(1.5pt) node[above]{$(v_2,1)$};
\filldraw(4,-2.5) circle(1.5pt) node[left]{$(v_0,0)$};
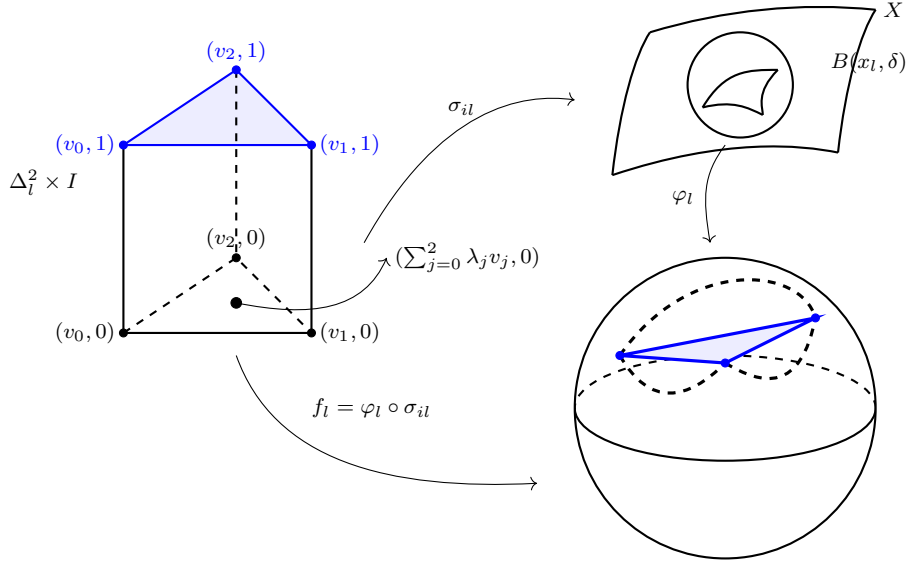
\filldraw(6.5,-2.5) circle(1.5pt) node[right]{$(v_1,0)$};
\filldraw(5.5,-1.5) circle(1.5pt) node[above]{$(v_2,0)$};
\filldraw(5.5,-2.1) circle(2pt);
\draw[->](5.5,-2.1)..controls(6.5,-2.3) and (7.2,-2.2)..(7.5,-1.5);
\draw(7.5,-1.5) node[right]{$(\sum_{j=0}^2\lambda_jv_j,0)$};
\draw(3.5,-0.5) node[left]{$\Delta_l^2\times I$};
\draw[thick](11,1.5)..controls(10.8,1.2) and (10.6,0.5)..(10.5,-.4);
\draw[thick](14,1.8)..controls(13.8,1.5) and (13.6,0.8)..(13.5,-.1);
\draw[thick](11,1.5)..controls(12,1.8) and (13,1.9)..(14,1.8);
\draw[thick](10.5,-.4)..controls(12,.1) and (13,0)..(13.5,-.1);
\draw[thick](12.2,.8) circle (.7);
\draw[thick](11.7,.5)..controls(12.1,.6) and (12.3,.5)..(12.5,.4);
\draw[thick](11.7,.5)..controls(12,1) and (12.3,1)..(12.7,1);
\draw[thick](12.7,1)..controls(12.4,.7) and (12.5,.5)..(12.5,.4);
\draw(14,1.8) node[right]{$X$};
\draw(13.3,1.1) node[right]{$B(x_l,\delta)$};
\draw[->](7.2,-1.3)..controls(8.2,.5) and (9.3,.7)..(10,.6);
\draw(8.5,.3) node[above]{$\sigma_{il}$};
\draw[thick](12,-3.5) circle(2);
\begin{scope}
\clip(10,-3.5) rectangle (14,-2.4);
\draw[thick, dashed] (12,-3.5) ellipse(2 and .7);
\end{scope}
\begin{scope}
\clip(10,-3.5) rectangle (14,-5);
\draw[thick] (12,-3.5) ellipse(2 and .7);
\end{scope}
\draw[very thick,dashed] (10.6,-2.8)..controls(11,-3.5) and (11.5,-3.4)..(12,-2.9);
\draw[very thick,dashed] (12,-2.9)..controls(12.5,-3.2) and (13,-3.3)..(13.2,-2.3);
\draw[very thick,dashed] (10.6,-2.8)..controls(11.5,-1.5) and (13,-1.6)..(13.2,-2.3);
\draw[fill,blue!7!white](10.6,-2.8)--(12,-2.9)--(13.2,-2.3)--cycle;
\draw[color=blue,very thick](10.6,-2.8)--(12,-2.9)--(13.2,-2.3)--(10.6,-2.8);
\filldraw[blue](10.6,-2.8) circle(1.5pt);
\filldraw[blue](12,-2.9) circle(1.5pt);
\filldraw[blue](13.2,-2.3) circle(1.5pt);
\draw[->](12,0)..controls(11.7,-0.4) and (11.7,-0.7)..(11.8,-1.3);
\draw(11.7,-0.7) node[left]{$\varphi_l$};
\draw[->](5.5,-2.8)..controls(6,-4.2) and (7.5,-4.6)..(9.5,-4.5);
\draw(6.4,-3.5) node[right]{$f_l=\varphi_l\circ\sigma_{il}$};
\end{tikzpicture}
\end{center}

Let $\beta_{il}=\varphi_l^{-1}\circ F_{il}:\Delta_l^n\times I\rightarrow X$ that can be seen as a homotopy equivalence between an $\epsilon$-continuous map and a continuous map. Since $\Delta_l^n\times I$ is a simplicial complex, $\beta_{il}$ can be regarded as a formal sum of $\epsilon$-singular $n+1$-simplices.

Finally, given $[\alpha]\in H_n(X^{\epsilon},X^{\mu})$, we can take $\beta=\sum_{i=1}^kn_i\sum_{l=1}^N\beta_{il}$ to be such that $\partial\beta$ is a formal sum of $\alpha$ as well as some $\mu$-continuous maps, i.e., $\alpha-\partial\beta\in C_n^{\mu}(X)$. Therefore, $H_n(X^{\epsilon},X^{\mu})=0$ and $H_n^{\mu}(X)\cong H_n^{\epsilon}(X)$ for $0<\mu\le\epsilon<\delta$. Actually, from the argument above, we can even see that $H_n^{\epsilon}(X)\cong H_n(X)$ for $0<\epsilon<\delta$. Therefore, by universal property of the inverse limit, $H_n^{lc}(X)\cong H_n(X)$ for all $n$.
\end{proof}

\begin{example}
Let $p$ be a point in $S^n$, then $H_n^{lc}(S^n\setminus\{p\})\cong\mathbb{Z}$ and $H_i^{lc}(S^n\setminus\{p\})=0$ for $i\neq n$. Indeed, $\overline{S^n\setminus\{p\}}=S^n$ and $S^n$ is a compact Riemannian manifold, so 
$$
H_i^{lc}(S^n\setminus\{p\})\cong H_i^{lc}(S^n)\cong H_i^n(S^n).
$$
\end{example}

\section{Topological entropy and lc-homology}

In this section, we give a generalization of the entropy conjecture and consequently establish a connection between topological entropy and $lc$-homology.

Let $(X,d)$ be a compact metric space with $\rank (H_i^{lc}(X))<\infty$ for each $i\geq 0$ and $f:X\rightarrow X$ be a continuous map. For the induced homomorphism $f^{lc}_{*i}: H^{lc}_i(X)\rightarrow H^{lc}_i(X)$, denote the spectral radius of $f^{lc}_{*i}$ by $\rho(f^{lc}_{*i})$, i.e.,
$$
\rho(f^{lc}_{*i})=\lim\limits_{n\rightarrow \infty}\|(f^{lc}_{*i})^n\|^{\frac{1}{n}}
$$
where $\|(f^{lc}_{*i})^n\|$ is the operator norm of the linear operator $(f^{lc}_{*i})^n$.

\begin{theorem}
$h(f)\geq \log \rho(f^{lc}_{*1})$.
\end{theorem}
\begin{proof}
Let $H_1^{lc}(X)=\varprojlim((H_1^{\epsilon_n}(X))_{n=0}^{\infty}, (\varphi_{\epsilon_k\epsilon_l *})_{l\leq k})$ and fix $\epsilon_0>0$. Then any loop of diameter $<\epsilon_0$ can be seen as contractible under $\epsilon_0$-scale. Choose $\mu<\epsilon_0$ such that whenever $d(x,y)<\mu$, $d(f(x),f(y))<\epsilon_0$. Then take $\delta<\frac{1}{3}\mu$ and clearly there exists an $\epsilon_0$-continuous path from $x$ to $y$ whenever $d(x,y)<3\delta$. Let $P$ be a $\delta$-net in $X$ and define $C_1^{\epsilon_0}(P,3\delta;\mathbb{R})$ to be the free abelian group generated by  $\{[x,y]:x,y\in P, d(x,y)<3\delta\}$, where $[x,y]$ denotes some fixed $\epsilon_0$-path from $x$ to $y$ of diameter $<\mu$, with real coefficients. Similarly, let $C_0^{\epsilon_0}(P;\mathbb{R})$ be the free abelian group generated by $P$ with real coefficients and we can define $\partial_{\epsilon_0}:C_1^{\epsilon_0}(P,3\delta;\mathbb{R})\rightarrow C_0^{\epsilon_0}(P;\mathbb{R})$ by $\partial_{\epsilon_0}[x,y]=y-x$. 

For any $\epsilon_1<3\delta$, consider the $\epsilon_0$-coordinate of  $\varprojlim((H_1^{\epsilon_n}(X))_{n=0}^{\infty}, (\varphi_{\epsilon_k\epsilon_l *})_{l\leq k})$ that is contained in $\varphi_{\epsilon_1\epsilon_0*}(H_1^{\epsilon_1}(X))$. 
Let $H_1^{\epsilon_1\epsilon_0}(P,3\delta;\mathbb{R})$ be the subgroup of $H_1^{\epsilon_0}(P,3\delta;\mathbb{R})$ whose representatives are in $C_1^{\epsilon_1}(P,3\delta;\mathbb{R})$
and it is naturally isomorphic to  $\varphi_{\epsilon_1\epsilon_0*}(H_1^{\epsilon_1}(X;\mathbb{R}))$. Indeed, every $\epsilon_0$-homology class in $\varphi_{\epsilon_1\epsilon_0*}(H_1^{\epsilon_1}(X;\mathbb{R}))$ has a representative in $C_1^{\epsilon_1}(P,3\delta;\mathbb{R})$ (defined similarly as $C_1^{\epsilon_0}(P,3\delta;\mathbb{R})$)  obtained by breaking down $\epsilon_1$-paths in $C_1^{\epsilon_1}(X;\mathbb{R})$ into combinations of short $\epsilon_1$-paths joining points of $P$, since $\epsilon_1<3\delta$. Then we can define a norm $\|\ \|$ on $C_1^{\epsilon_1}(P,3\delta;\mathbb{R})$ by 
$$\|\sum_i a_i\sigma_i\|=\sum_i|a_i|$$
and a norm $\|\ \|'$ on $H_1^{\epsilon_1\epsilon_0}(P,3\delta;\mathbb{R})$ by $$\|[u]\|'=\inf_{[\sigma]=[u]}\|\sigma\|$$
(since $H_1^{\epsilon_1\epsilon_0}(P,3\delta;\mathbb{R})$ is finite dimensional, all norms on it are equivalent).

Take $\epsilon_1$ to be sufficiently small. For a non-zero class $[u]\in H_1^{\epsilon_1\epsilon_0}(P,3\delta;\mathbb{R})$, take $\sigma=\sum_{i=1}^na_i\sigma_i\in C_1^{\epsilon_1}(P,3\delta;\mathbb{R})$ to be its cycle with $\|\sigma\|\le 2\|[u]\|'$. 

Let $Q_k$ be a minimal $(k,\delta)$-spanning set for $f$. For each $i$, we do the following construction with $\sigma_i:I\rightarrow X$. Define
$$
F_k=\id\times f\times f^2\times\cdots\times f^{k-1}: X\mapsto X\times f(X)\times\cdots\times f^{k-1}(X)
$$
and then the points of $F_k(Q_k)$ have $\delta$-neighbourhoods in $X^k$ (endowed with $d_{\infty}$ metric that takes the largest of the distances in each of the $k$ factors) that cover $F_k(X)$. For the sequence of these neighbourhoods through which $F_k(\sigma_i(I))$ passes, we pick a sequence $$x_1,x_2,\cdots,x_b\in Q_k$$
of some length b such that each $F_k(x_p)$ is in a $\delta$-neighbourhood of $F_k(\sigma_i(I))\subset X^k$, $d(f^j(x_{p-1}),f^j(x_p))<3\delta$ for $1<p\le b$, and $d(f^j(\sigma_i(0)),f^j(x_1)), d(f^j(x_b),f^j(\sigma_i(1)))<\delta$ for $0\le j<k$. Then $\sigma_i$ is homologous to 
$$
\tau_i=[\sigma_i(0),x_1]\ast [x_1,x_2]\ast [x_2,x_3]\cdots\ast[x_{b-1},x_b]\ast [x_b,\sigma_i(1)]
$$
where $\ast$ denotes the composition of paths.

It is possible that the length $b>|Q_k|$(the cardinality of $Q_k$). If $x_j=x_l$ for some $j\neq l$, then $\tau_i$ contains a loop from $x_j$ to $x_l=x_j$ and this loop is contained in an $\epsilon_1+\delta$ neighbourhood of $\sigma_iI$ which itself has diameter $<\epsilon_1$. So the diameter of the loop is $<\epsilon_0$. Hence we could suppress $x_{j+1},\cdots,x_l$ and assume that $b\le |Q_k|$.

Since $\sigma_i$ is $\epsilon_0$-continuous, without loss of generality, we can assume that $f^{m}\sigma_i$ is $\epsilon_0^{(m)}$-continuous with $\epsilon_0^{(k-1)}\geq\epsilon_0^{(k-2)}\geq\cdots\geq\epsilon_0^{(0)}=\epsilon_0$. We claim that under $\epsilon_0^{(k-1)}$-scale, $f^{k-1}\sigma_i$ is homologous ($\sim$) to $f^{k-1}\tau_i\sim v_i$, where 
$$
v_i=[f^{k-1}(\sigma_i(0)),f^{k-1}(x_1)]\ast\cdots\ast [f^{k-1}(x_{b-1}),f^{k-1}(x_b)]\ast [f^{k-1}(x_b),f^{k-1}(\sigma_i(1))]
$$

Indeed, $f[x_{j-1},x_j]\sim [f(x_{j-1}),f(x_j)]$ under $\epsilon_0^{(1)}$-scale since $f$ does not extend the path $[x_{j-1},x_j]$ so much. Similarly, $f[f(x_{j-1}),f(x_j)]\sim [f^2(x_{j-1}),f^2(x_j)]$ under $\epsilon_0^{(2)}$-scale etc. and so $f^{k-1}[x_{j-1},x_j]\sim [f^{k-1}(x_{j-1}),f^{k-1}(x_j)]$ under $\epsilon_0^{(k-1)}$-scale and 
$$
f^{k-1}\sigma\sim\sum a_if^{k-1}\sigma_i\sim \sum a_if^{k-1}\tau_i\sim \sum a_iv_i=v.
$$

Let $p:X\rightarrow P$ be a map such that $d(x,px)<\delta$ for $\forall x\in X$. Then we replace $\alpha$ by the $\epsilon_0^{(k-1)}$-singularly homologous cycle $\alpha$ obtained by replacing each $\epsilon_0^{(k-1)}$-singular simplex $[f^{k-1}(x),f^{k-1}(y)]$ in $v$ by $[pf^{k-1}(x),pf^{k-1}(y)]$ which is a generator of $C_1^{\epsilon_0^{(k-1)}}(P,3\delta;\mathbb{R})$ since $d(pf^{k-1}(x),pf^{k-1}(y))<3\delta$.

Since $f_{\epsilon_0*1}^{k-1}$ is the induced homomorphism on the $\epsilon_0$-coordinate of  $\varprojlim((H_1^{\epsilon_n}(X))_{n=0}^{\infty})$ to $H_1^{\epsilon_0^{(k-1)}}(X;\mathbb{R})$, by the choice of $\| \ \|'$, we have
$$\|f_{\epsilon_0*1}^{k-1}([u])\|'\le\|\alpha\|\le (1+|Q_k|)\sum|a_i|$$
whereas on $H_1^{\epsilon_1\epsilon_0}(P,3\delta;\mathbb{R})$,
$$\|[u]\|'\geq\frac{1}{2}\|\sigma\|=\frac{1}{2}\sum|a_i|.$$
Therefore 
$$
\frac{\|f_{\epsilon_0*1}^{k-1}([u])\|'}{\|[u]\|'}<2(1+|Q_k|)
$$
and this holds for any non-zero $[u]\in H_1^{\epsilon_0}(X;\mathbb{R})$ and all $k$. Hence $\|f_{\epsilon_0*1}^{k-1}\|'<2(1+|Q_k|)$. But $\rho(f_{\epsilon_0*1})=\lim\limits_{k\to\infty}\|f_{\epsilon_0*1}^k\|'^{\frac{1}{k}}$, so for the $\epsilon_0$-coordinate of $\varprojlim((H_1^{\epsilon_n}(X))_{n=0}^{\infty})$, 
\[\begin{split}
\log \rho(f_{\epsilon_0*1})&=\lim\frac{1}{k}\log \|f_{\epsilon_0*1}^k\|'\le \limsup\frac{1}{k}\log 2(1+|Q_{k+1}|)\\
&=\limsup\frac{1}{k}\log |Q_k|=h(f,\delta)\le h(f).
\end{split}
\]
Similarly, this will hold for each coordinate of the given inverse limit. Hence,
$$
\log \rho(f_{*1}^{lc})\le h(f).
$$
\end{proof}

\end{document}